%% file: main.tex
\documentclass[11pt,a4paper]{amsart}
\input{preamble.tex}

\makeatletter
\def\author@andify{%
  \nxandlist {\unskip ,\penalty-1 \space\ignorespaces}%
    {\unskip {} \@@and~}%
    {\unskip \penalty-2 \space \@@and~}%
}
\renewcommand{\andify}{%
  \nxandlist{\unskip, }{\unskip{} \@@and~}{\unskip \space \@@and~}}
\makeatother

\title[Metric-Free Riemannian Optimization]{Metric-Free Riemannian Optimization}

\author[J. Püschel]{Jonas Püschel}

\address[J. Püschel]{Institute of Mathematics, University of Augsburg, 
Universit\"atsstra{\ss}e~12a, 86159 Augsburg, Germany}
\email{jonas.pueschel@uni-a.de}

\thanks{Funded by the Deutsche Forschungsgemeinschaft (DFG, German Research Foundation) – 571768116.}

\begin{document}

\begin{abstract}
Riemannian optimization provides a powerful framework for constrained optimization by incorporating problem-specific structure directly into the geometry of the search space. In many applications, however, the explicit evaluation or application of the Riemannian metric can be computationally expensive or numerically unstable, limiting the practical efficiency of otherwise well-founded algorithms. Motivated by such settings, this work investigates to what extent classical Riemannian optimization algorithms can be reformulated without explicitly applying the metric.
We show that many first-order components of Riemannian optimization only rely  on the differential of the objective function and access to the Riemannian gradient, but not on explicit metric application. Based on this observation, we develop metric-free formulations and generalize optimization approaches to Finsler and Banach manifolds. 
Numerical experiments demonstrate that the proposed metric-free strategies retain the effectiveness of their metric-dependent counterparts while significantly reducing computational overhead. These results highlight that a substantial portion of Riemannian optimization can be carried out independently of explicit metric application, broadening its applicability to problems with expensive or implicitly defined metrics.
\end{abstract}

\maketitle



\section{Introduction}

Riemannian optimization is a type of optimization that does not take place in regular Euclidean space, but rather on a Riemannian manifold. It is useful, when the problem domain has some geometric or structural constraints, specifically on matrix manifolds \cite{Absil2007-ab}, as can be found in 
computer vision and signal processing \cite{ADMMS2002, DFVN2013}, low-rank matrix and tensor completion \cite{Kressner2014, Vandereycken2013}, quantum chemistry \cite{AltPS22, AltPS24, DaiLZZ17, PPS25} and quantum computing \cite{Le2025, Wiersema2023}. The main advantage of Riemannian optimization in these application is that instead of using Lagrange multipliers or projections, it is able to treat constraints as part of the geometry and thus preserve them by default. It also allows a more rigorous mathematical analysis of algorithms by using generalized results from unrestricted optimization, for example Riemmanian gradient descent (RGD) and Riemannian conjugate gradient (RCG). Especially for the latter one, a lot of metric applications are needed in order to calculate a sufficient step size and the the CG-parameter.

However in specific cases, the Riemannian metric is expensive to calculate explicitly and thus, applying the metric should be avoided if possible. This includes the energy-adaptive framework for nonlinear eigenvalue problems \cite{AltPS22, HerSY25, PPS25}, where evaluation of the metric equates to applying the Hamiltonian. In this paper we investigate how to reformulate algorithms in Riemannian optimization, such that applying the metric and calculation of the gradient at trial step sizes can be completely circumvented. A lot of first-order quantities in Riemannian optimization do not depend on the specific 
metric, but only on the differential of the objective function. We show that as long as one has access to the differential and the Riemannian gradient, one can formulate descent conditions, line search conditions, step size algorithms and algorithms like Riemannian conjugate gradient completely without ever explicitly applying the metric. This also highlights the fact that a lot of quantities are metric-independent and could in theory be used in a Finsler or even Banach manifold setting. 

The remaining part of this paper is structured as follows. In Section 2, we recall the basics of Riemannian optimization and necessary tools. In Section 3 we discuss the well-known symmetric preconditioned conjugate gradient method (PCG) for linear symmetric positive definite matrices as an example where expensive inner product application can be circumvented. Section 4 extends on this concept of by showing that a multitude of ingredients for Riemannian optimization methods can be calculated without applying the Riemannian metric, which we refer to as metric-free. In Section 6, we discuss extension of these quantities to Finsler and Banach manifolds. Effectiveness of metric-free strategies is illustrated via numerical experiments in Section 6.

\section{Preliminaries}

This section briefly introduces the fundamental concepts of Riemannian optimization, basic Riemannian geometry on the one hand and tools of manifold optimization on the other hand. Although in practice one primarily uses finite-dimensional manifolds \cite{Absil2007-ab}, the following quantities can also be defined in the infinite-dimensional case \cite{Lang1999}.

\subsection{Riemannian Gradient}

A \emph{Riemannian manifold} is a tuple  $(\M, \langle \cdot, \cdot \rangle )$ with smooth manifold $\M$ and {Riemannian metric} $\langle \cdot, \cdot \rangle$. A \emph{Riemannian metric} (or short metric) is an inner product on the tangent space smoothly depending on $x \in \M$. Specifically, $\langle \cdot, \cdot\rangle_x$ induces an inner product on $\Tan_x \M$ for all $x \in \M$. For infinite dimensional Riemannian manifolds, we require the tangent space to be complete with respect to the induced norm, meaning $(\Tan_x\M, \langle \cdot, \cdot \rangle_x)$ is a Hilbert space. In literature, this is called a \emph{strong Riemannian metric} \cite{Lang1999}. In contrast, if the tangent space is not complete with respect to the induced norm, we only have a \emph{weak Riemannian metric}. In this stting, it is still possible to do Riemannian optimization, however additional care needs to be taken \cite{ZalPS26}.

For a differentiable function $f \colon \M \to \IR$, the \emph{Riemannian gradient} of $f$ at $x \in \M$, $\grad f(x)$ is defined as the Riesz representative of the derivative, meaning for all $v \in \Tan_x\M$
\begin{equation}\label{eq:riem_grad}
\langle \grad f(x), v\rangle_x = \Drm f(x)[v].
\end{equation}
Here the Riesz representation theorem guarantees existence and uniqueness. This highlights the need of a (strong) Riemannian metric, as otherwise existence of the gradient cannot be guaranteed without additional assumptions. Also, different metrics induce different Riemannian gradients. 

\subsection{Retractions}

In a Euclidean space $E$, a step of a line search-based method is given by
\begin{equation}\label{eq:grad_desc}
    x_{k+1} = x_k + \alpha_k d_k
\end{equation}
for an iterate $x_k \in E$, a search direction $d_k \in E$ and a step size $\alpha_k > 0$.
However, this no longer makes sense in the framework of Riemannian optimization. Even if addition is well-defined (which is only the case for submanifolds of vector spaces), there is no guarantee that the resulting point is still on the manifold. Thus, Riemannian optimization needs a more general way to perform steps into tangential directions. 

A \emph{retraction} is a smooth mapping that approximates the exponential map on the Riemannian manifold $(\M,\langle \cdot, \cdot \rangle)$ (which with respect to $\langle \cdot, \cdot \rangle$ is the true geometric way to move along geodesics on a manifold). In simple terms, the retraction takes a step in a tangent direction from a point on the manifold and maps it back to the manifold. More formally, a \emph{retraction} $\R$ is a smooth map
\begin{equation}\label{eq:retraction}
    \begin{aligned}
        \R_x \colon \Tan\M &\to \M, \\
        (x,v) &\mapsto \R_x(v)
    \end{aligned}
\end{equation}
that fulfills the following conditions at all $x \in \M$:
\begin{itemize}
    \item There is no movement when taking a step into the zero direction, meaning
    \begin{equation}\label{eq:ret1}
        \R_x(0) = x.
    \end{equation}
    \item The initial direction is preserved, meaning $\R_x$ yields a first-order approximation of the exponential map, meaning for all $v \in \Tan_x\M$
    \begin{equation}\label{eq:ret2}
        \frac{\drm}{ \drm t} \R_x(tv) \Bigg|_{t=0} = v.
    \end{equation}
\end{itemize}
A retraction allows to perform steps in arbitrary tangent directions. The Riemannian reformulation of the iteration \eqref{eq:grad_desc} thus reads
\begin{equation}\label{eq:riem_grad_desc}
    x_{k+1} = \R_{x_k}(\alpha_k d_k)
\end{equation}
for an iterate $x_k \in \M$, search direction $d_k \in \Tan_{x_k}\M$ and step size $\alpha_k > 0$. 



\subsection{Transports} 
Momentum-based methods, like the Riemannian CG 
accumulate gradients or directions across iterations. However, each gradient or direction lives in a different tangent space. Thus, it is necessary to transport these quantities into the tangent space of the current iterate. For a given retraction $\R$ a {vector transport} $\T$  provides a way to transport tangent vectors alongside $\R$. Formally, we call $\T$ a \emph{vector transport} for the associated retraction $\R$ if the following conditions hold:
\begin{itemize}
    \item For all $x \in \M$ and $v, w \in \Tan_x\M$, $\T_v$ transports quantites from $\Tan_{x}\M$ to $\Tan_{\R_x(v)}\M$, thus
    \begin{equation}\label{eq:transport1}
        \T_v(w) \in \Tan_{\R_x(v)}\M.
    \end{equation}
    \item If no step is taken, the transport is the identity, meaning for all $x \in \M$ and $0_v \in \Tan_x\M$ being the zero-element we have
    \begin{equation}\label{eq:transport2}
        \T_{0_v} = \mathrm{id}_{\Tan_x\M}.
    \end{equation}
    \item The transport is homogeneous, meaning for all $x \in \M$, $v, w \in \Tan_x\M$ and $\lambda\in\mathbb R$ holds 
        \begin{equation}\label{eq:transport3}
        \T_v(\lambda w) =  \lambda\T_v(w).
    \end{equation}
\end{itemize}
Some literature, notably \cite[Chapter 8]{Absil2007-ab}, additionally demands $\T_v$ to be a linear map. However, the commonly used scaled transport introduced in \cite{RiWi2012} is only homogeneous, not linear.

A natural way to construct a vector transport associated to a given retraction $\R$ is the differentiated retraction transport, which for $x \in \M$ and $v,w \in \Tan_x\M$ is given by
\begin{equation}\label{eq:diff_ret}
    \T_v^{\R}(w) \coloneqq \Drm \R_x(v)[w],
\end{equation}
which is the directional derivative of $\R_x$ at $v$ in direction $w$. 
One can easily verify that the differentiated retraction transport fulfills \eqref{eq:transport1}, \eqref{eq:transport2} and \eqref{eq:transport3} in addition to being linear. In the special case of $\R$ being the exponential map, the differentiated retraction transport corresponds to the parallel transport by construction. 

\section{A Euclidean Example}
In this section, we discuss the symmetric preconditioned conjugate gradient method for solving symmetric positive definite linear systems as a prototypical example for a method where applying an expensive inner product can be circumvented. Notably, this example also reveals that preconditioning and change of inner product (in the Rimannian case metric) are very closely related and can be viewed as equivalent. 

Let $A, M\in \mathbb R^{n\times n}$ be symmetric positive definite (s.p.d.) matrices with $M \approx A$, and let $b \in \mathbb R^n$. We are looking for $x \in \mathbb R^n$ such that $Ax=b$ via minimizing the energy functional 
\begin{equation}\label{eq:cg_energy}
    \E(x) = \frac{1}{2} x^TAx - x^Tb.
\end{equation}
For that, we formulate the conjugate gradient method with respect to the inner product
\begin{equation}\label{eq:inner_M}
    \langle v,w \rangle_M \coloneqq v^T M w    
\end{equation}
and with initial guess $x_0 \in \mathbb R^n$. We first show how to calculate the gradient of $\E$ in this inner product. 
\begin{lem}\label{lem:1}
    Let $x \in \mathbb R^n$. The gradient of $\E$ at $x$ with respect to $\langle \cdot, \cdot\rangle_M$ is given by 
    $$\nabla_M \E(x) = M^{-1}(Ax-b).$$
\end{lem}
\begin{proof}
    Let $v \in \mathbb R^n$. We have
    \begin{align*}
        \Drm \E(x)[v] &= \frac{1}{2}(x^TAv + v^TAx) - v^Tb 
                \\&= v^T(Ax-b)
                \\&= v^T M M ^{-1} (Ax-b)
                \\&= \langle v, M ^{-1} (Ax-b)\rangle_M.
    \end{align*}
\end{proof}
When defining the residual (meaning Euclidean gradient) at $x$ via $r = Ax-b$ and the $\langle \cdot, \cdot\rangle_M$-gradient $g = M^{-1}r$, we have the following equivalent expressions for $\Drm \E(x)[v]$ for any $v \in \mathbb R^n$
\begin{equation}\label{eq:cg_metric_equality}
    \Drm \E(x) [v] = \langle r, v\rangle = \langle g, v\rangle_M
\end{equation}
where $\langle v,w\rangle \coloneqq v^Tw$ denotes the Euclidean inner product on $\mathbb R^n$. 
We also give a formula for the Hessian with respect to the inner product induced by $M$.
\begin{lem}\label{lem:2}
    Let $x\in \mathbb R^n$. The Hessian of $\E$ with respect to $\langle \cdot, \cdot\rangle_M$ is given by 
    $$\nabla^2_M\E(x) = M^{-1}A.$$
\end{lem}
\begin{proof}
    Let $v,w \in \mathbb R^n$. We have
    \begin{align*}
        \Drm^2\E(x) [v,w]&= \frac{1}{2}(v^TAw + w^TAv) 
                \\&= v^TAw \\&= 
                v^T MM^{-1}Aw
                \\&= \langle v, M^{-1}A w \rangle_M
    \end{align*}
\end{proof}

The $\langle \cdot, \cdot\rangle_M$-conjugate gradient (CG) method with Fletcher--Reeves CG-parameter $\beta_k$\cite{FR64} and Hessian step size $\alpha_k$ is given in Algorithm \ref{alg:cg1}. Using Lemmas \ref{lem:1} and \ref{lem:2} on the expressions in lines 5 and 9 yields the well-known classical preconditioned conjugate gradient method (PCG) in Algorithm \ref{alg:cg2}. As can be seen, the possibly expensive evaluations of the inner product in lines 5 and 9 can be completely circumvented. Consequently, CG with respect to the $M$-inner product can be executed without ever explicitly evaluating $\langle \cdot, \cdot \rangle_M$.

\begin{figure*}[ht]

\begin{minipage}{0.48\textwidth}
\small
\begin{algorithm2e}[H]
\SetKwInOut{Input}{input}
\SetAlgoLined
\DontPrintSemicolon 
\Input{s.p.d. matrices $A,M \in \mathbb R^{n \times n}$, $b \in \mathbb R^n$, init. guess $x_0 \in \mathbb R^n$}
\BlankLine
$r_0 = Ax_0- b$\;
$g_0 = M^{-1} r_0$\;
$d_0 = -g_0$ \;
 \For{$k=0, 1, 2, \dots$ until convergence}{
    $\alpha_k = -\frac{\langle g_k, d_k\rangle_M}{\langle d_k, M^{-1}A d_k\rangle_M} \phantom{\frac{ r_k^T d_k}{d_k^TA d_k}}$  \; 
    $x_{k+1} = x_k + \alpha_k d_k$\;
    $r_{k+1} = r_k + \alpha_k Ad_k$\;
    $g_{k+1} = M^{-1} r_{k+1}$ \;
    $\beta_{k} = \frac{\langle g_k, g_k\rangle_M}{\langle g_{k+1}, g_{k+1}\rangle_M} \phantom{\frac{r_k^T g_k}{r_{k+1}^Tg_{k+1}}}$\;
    $d_{k+1} = - g_{k+1} + \beta_k d_k$\;
 }
 \KwRet{$x_{k+1}$}
 \caption{$\langle \cdot, \cdot \rangle_M$-CG}
     \label{alg:cg1}
\end{algorithm2e}
\end{minipage}\hfill
\begin{minipage}{0.48\textwidth}
\small
\begin{algorithm2e}[H]
\SetKwInOut{Input}{input}
\SetAlgoLined
\DontPrintSemicolon 
\Input{s.p.d. matrices $A,M \in \mathbb R^{n \times n}$, $b \in \mathbb R^n$, init. guess $x_0 \in \mathbb R^n$}
\BlankLine
$r_0 = Ax_0- b$\;
$g_0 = M^{-1} r_0$\;
$d_0 = -g_0$ \;
 \For{$k=0, 1, 2, \dots$ until convergence}{
    $\alpha_k = -\frac{ r_k^T d_k}{d_k^TA d_k}$ \;
    $x_{k+1} = x_k + \alpha_k d_k$\;
    $r_{k+1} = r_k + \alpha_k Ad_k$\;
    $g_{k+1} = M^{-1} r_{k+1}$ \;
    $\beta_{k} = \frac{r_k^T g_k}{r_{k+1}^Tg_{k+1}}$ \;
    $d_{k+1} = - g_{k+1} + \beta_k d_k$\;
 }
 \KwRet{$x_{k+1}$}
 \caption{PCG \phantom{$\langle \cdot, \cdot \rangle_M$}}
     \label{alg:cg2}
\end{algorithm2e}

\end{minipage}

\end{figure*}

In other words, preconditioning and change of the inner product are equivalent: Preconditioning with $M^{-1}$ is mathematically equivalent to using the iner product $\langle \cdot, \cdot \rangle_M$ for the linear conjugate gradient method in $\mathbb R^n$. 

PCG from algorithm \ref{alg:cg2} illustrates the exact setting in which metric-free approaches. We have
\begin{itemize}
    \item a ``cheap'' inner product $\langle v, w\rangle = v^Tw$ to calculate the relevant quantities like step size and CG-parameter and
    \item an ``expensive'' inner product $\langle v, w\rangle_M = v^TMw$ that is only used for the gradient and never applied explicitly. 
\end{itemize}
Using this approach allows to use the gradient with respect to a better-performing inner product without the overhanging cost for calculation of step sizes and other parameters -- the expensive inner product $\langle \cdot, \cdot \rangle_M$ is never applied explicitly. Additionally, it reduces numerical errors in CG-parameter and step size, since usually either gradient calculation or inner product application involve a matrix inversion, which usually is solved inexactly by an iterative scheme.

\section{Metric-Free Riemannian Optimization}

In this section we show that the result above is also applicable to Riemannian optimization methods that leverage first-order quantities.  We first discuss the general idea that enables metric-free Riemannian optimization and then adopt well-known quantities like CG parameters and step size control into the metric-free framework. We then explicitly state how to transfer quantities from different Riemannian approaches into the metric-free setting. As a rule of thumb, if quantities only make use of first-order information, transferability to the metric-free framework is possible.  

\subsection{General Idea}

Let $(\M,\langle \cdot, \cdot \rangle)$ be a Riemannian manifold and $f: \M \to \mathbb R$ be a differentiable function. 
We recall the definition of the Riemannian gradient from \eqref{eq:riem_grad} for arbitrary $x \in \M$, $v \in \Tan_x\M$
\begin{equation*}
    \langle \grad f(x), v\rangle_x = \Drm f(x) [v].
\end{equation*}
The general idea of metric-free Riemannian optimization is to formulate all metric evaluations in this way, consequently making it possible to calculate them via evaluating the differential.
This is exactly the equation given from the definition of the Riemannian gradient \eqref{eq:riem_grad}. In order to make sense of metric-free Riemannian optimization, we however flip the interpretation: Quantities that are usually expressed as the metric applied to the Riemannian gradient and any tangent vector can instead be expressed as a directional derivative. Specifically, these quantities are independent of the metric chosen, and the only first-order quantity that directly depends on the metric is the gradient itself. Leveraging this, we are able to reformulate quantities from Riemannian optimization without applying the metric, resulting in \emph{metric-free} algorithms, meaning algorithms free of metric applications. Obviously, metric-free Riemannian algorithms are not metric independent, since the Riemannian gradient and other quantities still depend on the metric. In Section 5, we will briefly discuss how optimization on non-Riemannian manifolds, meaning without existance of a Riemannian metric, can look.

In order to formulate metric-free quantities, we need access to the Riemannian gradient $\grad f(x)$ and the differential $\Drm f(x)$. A possible setting in practice is closely related to the case discussed in the previous section. We have a manifold $\M$ with 
\begin{itemize}
    \item a ``cheap'' Riemannian metric $\langle \cdot, \cdot \rangle_a$ where both metric application and gradient calculation are comparatively cheap and
    \item an ``expensive'' Riemannian metric $\langle \cdot, \cdot \rangle_b$ where both metric application and gradient calculation are comparatively expensive.
\end{itemize}
This is especially the case, when we have 
$$ \langle v,w \rangle_{b,x} = \langle M_x(v), w \rangle_{a,x}\qquad \text{for all }x\in \M, v,w \in \Tan_x\M$$
where $M_x : \Tan_x\M \to \Tan_x\M$ is a linear, symmetric (with respect to $\langle \cdot, \cdot \rangle_{a,x}$), positive definite operator smoothly depending on $x$. Then, applying $\langle \cdot, \cdot \rangle_{b,x}$ costs one application of $\langle \cdot, \cdot \rangle_{a,x}$ and one application of $M_x$, and for a differentiable $f \colon \M \to \mathbb R$ the gradient calculation of $\grad_b f(x)$ costs a gradient calculation of $\grad_a f(x)$ and one inversion of $M_x$. Consequently, the derivative of $f \colon \M \to \mathbb R$ can be expressed via
$$\Drm f(x) [v] = \langle \grad_a f(x), v \rangle_{a,x},$$
which is far cheaper and possibly more numerically accurate than
$$\Drm f(x) [v] = \langle \grad_b f(x), v \rangle_{b,x} = \langle M_x M_x^{-1} \grad_a f(x), v \rangle_{a,x}.$$
Especially when $\grad_b f(x)$ is not used for any further purposes, the latter formula generates a lot of unnecessary computational overhead over the former, and metric-free approaches allow us to avoid such unnecessary and expensive operations.
Such a setting will later be revisited in the numerical experiments.

In the following, we label the equations for the expression as
\begin{itemize}
    \item (XXa) for the Euclidean formulation;
    \item (XXb) for the usual Riemannian formulation;
    \item (XXc) for the metric-free formulation.
\end{itemize}
In Euclidean spaces, {(XXa)}, {(XXb)} and {(XXc)} are equivalent. The word ``usual'' here implies the way how the expression can be and is usually generalized. However, there does not exist a strictly canonical way of generalizing expressions to the Riemannian framework. We give a brief example to illustrate this.
\begin{exm}\label{ex:1}
\begin{subequations}
Consider an iterative method in $\mathbb R^n$ with gradients $g_k$ and search directions $d_k$ at iterates $x_k$, where the quantity
\begin{equation} \label{eq:euc_remark}
    \langle y_{k+1}, d_k \rangle, \qquad y_{k+1} = g_{k+1} - g_k
\end{equation}
is calculated. This appears for example in the Hestenes--Stiefel CG parameter \emph{\cite{HS52}}. The direct Riemannian generalization reads
\begin{equation} \label{eq:riem_remark}
    \langle y_{k+1}, \T(d_k) \rangle_{x_{k+1}}, \qquad y_{k+1} = g_{k+1} - \T (g_k),
\end{equation}
with vector transport $\T \colon \Tan_{x_{k}}\M \to \Tan_{x_{k+1}}\M$. This is a straight-forward translation of \emph{\eqref{eq:euc_remark}}. The original expression is optimized to use as few inner products as possible, which is useful in Euclidean space, since it reduces cost. However, this approach may be flawed in the Riemannian case, because it not only prevents metric-free evaluation, but also forces us to perform two vector transports. The expression \eqref{eq:riem_remark} contains the term $\langle \T (g_k), \T (d_k) \rangle_{x_{k+1}}$, where two directions are transported into the tangent space at $x_{k+1}$, just to take the metric there. Directly calculating the metric at $x_k$ seems to be a less error prone approach, especially since it is known that (unscaled) vector transports can be problematic \emph{\cite{RiWi2012}}. Using the expression equivalent to \emph{\eqref{eq:euc_remark}}
\begin{equation*}
    \langle g_{k+1}, d_k\rangle - \langle g_{k}, d_k\rangle
\end{equation*}
one arrives at a more natural Riemannian formulation
\begin{equation} \tag{14b'} \label{eq:riem_remark_2}
     \langle g_{k+1}, \T(d_k)\rangle_{x_{k+1}} - \langle g_{k}, d_k\rangle_{x_{k}},
\end{equation}
 which can also be found in \emph{\cite{Sato22}}. We note that \eqref{eq:riem_remark} and \eqref{eq:riem_remark_2} are in general not equivalent. They however coincide if an isometric transport, for example the parallel transport is used.  
This contains only  one vector transport and can also be formulated metric-free via
\begin{equation}
    \Drm f(x_{k+1})[\T(d_k)] - \Drm f(x_{k})[d_k].
\end{equation}
\end{subequations}
\end{exm}

As the example shows, the naive formulations of quantities in Riemannian optimization are obtained by Euclidean expressions minimizing the amount of inner product applications, which is not necessarily the best way of generalizing it to Riemannian optimization. Since metric-free expressions give a more natural formulation of those quantities, by minimizing the amount of transports that need to be performed, one could argue that those should be preferred. 

We also give a short remark, why second-order quantities are in genereal not generalizable to the metric-free framework.

\begin{rem}[Second-order quantities] \label{rem:hess}
    Second-order quantities, like the Riemannian Hessian and its evaluation, do indeed directly depend on the metric, since the curvature of the manifold does affect those. Thus, methods relying on second-order quantities in addition to first-order, like Riemannian Newton methods or step size control relying on the Hessian, can in general not be formulated metric-free. In case of the example from the previous section, we get transferability of the Hessian, since both $(\mathbb R^n, \langle \cdot, \cdot\rangle)$ and $(\mathbb R^n, \langle \cdot, \cdot \rangle_M)$ as Riemannian manifolds have constant curvature $0$. 
\end{rem}

\subsection{Search direction conditions} We examine conditions for the search direction $d_k \in \Tan_{x_k}\M$ at an iterate $x_k \in \M$.

\paragraph{Descent condition} The basic descent condition in the Euclidean case takes the form
\begin{subequations}
\begin{align}
    \label{eq:euc_desc_cond}
    \nabla f(x_k)^T d_k &< 0, \\
    \label{eq:riem_desc_cond}
    \langle \grad f(x_k), d_k\rangle_{x_k} &< 0, \\
    \label{eq:mf_desc_cond}
    \Drm f(x_k)[d_k] &< 0.
\end{align}
\end{subequations}
Transferability into the metric-free framework is a natural result, since the descent direction property of $d_k$ is independent of the metric.

\paragraph{Sufficienct descent condition} The sufficient descent direction, sometimes also called strong or uniform descent condition, is a stronger version of the descent condition. For a constant $c > 0$ it reads
\begin{subequations}
\begin{align}
    \label{eq:euc_sd_cond}
    \nabla f(x_k)^T d_k &\le -c \|\nabla f(x_k)\|^2, \\
    \label{eq:riem_sd_cond}
    \langle \grad f(x_k), d_k\rangle_{x_k} &\le -c\|\grad f(x_k)\|_{x_k}^2, \\
    \label{eq:mf_sd_cond}
    \Drm f(x_k)[d_k] &\le -c \, \Drm f(x_k)[\grad f(x_k)].
\end{align}
\end{subequations}

\paragraph{Angle condition} The angle condition bounds the angle between the negative gradient and search direction from below via $\cos(\angle (-\grad f(x_k), d_k))  \ge c $
for a given constant $1\ge c > 0$. It reads 
\begin{subequations}
\begin{align}
    \label{eq:euc_angle_cond}
    \frac{-\nabla f(x_k)^Td_k}{\|\nabla f(x_k)\|\,\|d_k\|} &\ge c, \\
    \label{eq:riem_angle_cond}
    \frac{-\langle \grad f(x_k), d_k \rangle_{x_k}}{\|\grad f(x_k)\|_{x_k}\,\|d_k\|_{x_k}} &\ge c, \\
    \label{eq:mf_angle_cond}
    \frac{- \Drm f(x_k) [d_k]}{\sqrt{\Drm f(x_k) [\grad f(x_k) ]}\,{\|d_k\|_{x_k}}} &\ge c.
\end{align}
\end{subequations}
Since the metric plays a integral part in calculating angles, a true metric-free formulation of the angle condition is not possible.
Indeed in \eqref{eq:mf_angle_cond} evaluation of the metric is needed for $\|d_k\|_{x_k} = \sqrt{\langle d_k, d_k\rangle_{x_k}}$. However, if there is a way to calculate $\|d_k\|_{x_k}$ without applying the metric, the angle condition can be checked metric-free. This is explored in Section 5.2. 

\subsection{Line search conditions}\label{sec:linesearch_conds} Line search conditions are conditions imposed on the step size $\alpha$ with given iterate $x_k \in \M$ and a search direction $d_k \in \Tan_{x_k}\M$ that, to ensure well-definedness is a descent direction. In this context, we also consider a retraction $\R$ and a vector transport $\T$.

\paragraph{Armijo condition}  The Armijo condition, or sufficient descent condition, is the most basic line search condition. It ensures that the value of the cost function decreases sufficiently. For a constant $0 < c_1 < 1$ it reads
\begin{subequations}\label{eq:armijo}
\begin{align}
    \label{eq:euc_armijo}
    f(x_k + \alpha d_k) &\le f(x_k) + c_1 \alpha \, \nabla f(x_k) ^T d_k, \\
    \label{eq:riem_armijo}
    f(\R_{x_k}(\alpha d_k)) &\le f(x_k) + c_1 \alpha  \langle \grad f(x_k), d_k\rangle_{x_k},\\
    \label{eq:mf_armijo}
    f(\R_{x_k}(\alpha d_k)) &\le f(x_k) + c_1 \alpha  \Drm f(x_k)[d_k].
\end{align}  
\end{subequations}

\paragraph{Goldstein condition}
The Goldstein condition imposes a lower bound on the cost function decrease in order to avoid steps that are too small. For a constant $0 < c_1 < \frac{1}{2}$, it reads
\begin{subequations}\label{eq:goldstein}
\begin{align}
        \label{eq:euc_goldstein}
    f(x_k + \alpha d_k) &\ge f(x_k) + (1 - c_1) \alpha \, \nabla f(x_k) ^T d_k, \\
    \label{eq:riem_goldstein}
    f(\R_{x_k}(\alpha d_k)) &\ge f(x_k) + (1 - c_1) \alpha  \langle \grad f(x_k), d_k\rangle_{x_k},\\
    \label{eq:mf_goldstein}
    f(\R_{x_k}(\alpha d_k)) &\ge f(x_k) + (1 - c_1) \alpha  \Drm f(x_k)[d_k].
\end{align}
\end{subequations}
The Goldstein condition is usually employed in conjunction with the Armijo condition, both using the same $c_1$. 

\paragraph{Curvature conditions} The curvature conditions state that the slope of $f$ along the trajectory $\alpha \mapsto x_k + \alpha d_k$ (in the Euclidean case) or $\alpha \mapsto \R_{x_k}(\alpha d_k)$ (in the Riemannian case) decreases sufficiently in order to approximate a first-order optimal step size. In the Riemannian case, we need to transport the search direction $d_k$ alongside the trajectory, for which we need the vector transport $\T$. We note that for mathematical consistency with the aforementioned concept, one needs to use the differentiated retraction \eqref{eq:diff_ret}, since then 
\begin{equation}\label{eq:transport_diffret_curv}
    \frac{\drm}{\drm t} (f(\R_{x_k}(t d_k))) = \Drm f(\R_{x_k}(t d_k))[\Drm \R_{x_k}(t d_k)[d_k]] = \Drm f(\R_{x_k}(t d_k))[\T_{t d_k}(d_k)].
\end{equation}
For a constant $0<c_2<1$, the weak curvature condition reads
\begin{subequations}\label{eq:weak_curv}
    \begin{align}
        \label{eq:euc_weak_curv}
        \nabla f(x_k + \alpha d_k)^Td_k &\ge c_2 \nabla f(x_k)^T d_k,\\
        \label{eq:riem_weak_curv}
        \langle \grad f( \R_{x_k} (\alpha d_k)), \T_{\alpha d_k}(d_k) \rangle_{\R_{x_k} (\alpha d_k)} &\ge c_2 \langle \grad f(x_k), d_k \rangle_{x_k}, \\
        \label{eq:mf_weak_curv}
        \Drm f(\R_{x_k} (\alpha d_k))[\T_{\alpha d_k}(d_k)] &\ge c_2 \Drm f(x_k) [d_k].
    \end{align}
\end{subequations}
The strong curvature condition reads
\begin{subequations}\label{eq:strong_curv}
\begin{align}
    \label{eq:euc_strong_curv}
    \nabla f(x_k + \alpha d_k)^Td_k &\le c_2 |\nabla f(x_k)^T d_k|, \\
    \label{eq:riem_strong_curv}
    \langle \grad f( \R_{x_k} (\alpha d_k)), \T_{\alpha d_k}(d_k) \rangle_{\R_{x_k} (\alpha d_k)} &\le c_2 |\langle \grad f(x_k), d_k \rangle_{x_k}|, \\
    \label{eq:mf_strong_curv}
        \Drm f(\R_{x_k} (\alpha d_k))[\T_{\alpha d_k}(d_k)] &\le c_2 |\Drm f(x_k) [d_k]|.
\end{align}
\end{subequations}

When combining the weak/strong curvature condition with the Armijo condition for $0 < c_1 < c_2 < 1$, one obtains the weak/strong Wolfe conditions. In the Riemannian optimization literature, Wolfe conditions are sometimes already formulated in a metric-free fashion, e.g. in \cite{RiWi2012}.

Both \eqref{eq:mf_weak_curv} and \eqref{eq:mf_strong_curv} can be very useful in cases where backtracking-based line searches are employed. Not only do they allow to check the conditions without evaluation of the metric, they also avoid calculation of the gradient, which generally is very costly.

\subsection{Step size strategies} Step size strategies are algorithms that return a step size $\alpha$ for given iterate $x_k \in \M$ and search direction $d_k \in \Tan_{x_k}\M$. We differentiate between heuristics that directly return a step size, and iterative backtracking-like schemes that calculate a step size such that some line search conditions, as described in the previous section, are met. 

\paragraph{Polyak step size} The Polyak step size \cite{Polyak69B} is a popular gradient-only step size strategy in gradient and subgradient methods for (geodesic) convex optimization. Given the exact value or an estimate of the minimum $f^\star$ of $f$, it reads
\begin{subequations}
\begin{align}
    \label{eq:euc_polyak}
        \alpha_k &= \frac{f(x_k) - f^\star}{\|\nabla f(x_k)\|^2}, \\
        \label{eq:riem_polyak}
        \alpha_k &= \frac{f(x_k) - f^\star}{\|\grad f(x_k)\|^2_{x_k}}, \\
        \label{eq:mf_polyak}
        \alpha_k &= \frac{f(x_k) - f^\star}{\Drm f(x_k) [\grad f(x_k)]}.
\end{align}
\end{subequations}

\paragraph{Barzilai-Borwein step size} The Barzilai-Borwein (BB) step size strategy \cite{BB1988} is also a gradient-only heuristic that calculates a step size by approximating second-order information by tracking how gradients change over iterations. It is very popular for gradient methods and oftentimes performs better than the exact line search. It alternates between a short and a long step and takes an initial step size $\alpha_0$. Here we only consider (Riemannian) gradient descent and since we always transport from the tangent space at $x_{k-1}$ to the one at $x_k$, we use the notation $\T = \T_{-\alpha_{k-1}\grad f(x_{k-1})}$. The BB step size reads
\begin{subequations}

    \begin{align}
        \label{eq:euc_bb}
        &\begin{aligned}
            \alpha_k^{\rm long} &= \frac{(x_{k} - x_{k-1})^T(x_{k} - x_{k-1})}{(x_{k} - x_{k-1})^T(\nabla f(x_{k}) - \nabla f(x_{k-1}))} ,\\
            \alpha_k^{\rm short} &= \frac{(x_{k} - x_{k-1})^T(\nabla f(x_{k}) - \nabla f(x_{k-1}))}{(\nabla f(x_{k}) - \nabla f(x_{k-1}))^T(\nabla f(x_{k}) - \nabla f(x_{k-1}))},
        \end{aligned} \\
        \label{eq:riem_bb}
        &\begin{aligned}
            \alpha_k^{\rm long} &=  \frac{\alpha_{k-1}\|\grad f(x_k)\|^2_{x_{k-1}}}{\langle \T(\grad f(x_{k-1})), \T(\grad f(x_{k-1})) - \grad f(x_k)   \rangle_{x_k}} ,\\
            \alpha_k^{\rm short} &= \frac{\alpha_{k-1}\langle \T(\grad f(x_{k-1})), \T(\grad f(x_{k-1})) - \grad f(x_k)   \rangle_{x_k}}{\langle \grad f(x_k) -  \T(\grad f(x_{k-1}) , \grad f(x_k) -  \T(\grad f(x_{k-1}))   \rangle_{x_k}},
        \end{aligned} \\
        \label{eq:mf_bb}
        &\begin{aligned}
            \alpha_k^{\rm long} &=  \frac{\alpha_{k-1}\Drm f(x_{k-1})[\grad f(x_{k-1})]}{\Drm f(x_{k-1})[\grad f(x_{k-1})] - \Drm f(x_{k})[\T (\grad f(x_{k-1}))]} ,\\
            \alpha_k^{\rm short} &= \frac{\alpha_{k-1}\Drm f(x_{k-1})[\grad f(x_{k-1})] - \Drm f(x_{k})[\T (\grad f(x_{k-1}))]}{\Drm f(x_{k-1})[\grad f(x_{k-1})] + \Drm f(x_{k})[\grad f(x_{k}) - 2\T (\grad f(x_{k-1})) ]},
        \end{aligned}
    \end{align}
\end{subequations}
where we substituted $x_{k} - x_{k-1}$ by $-\alpha_{k-1} \grad f(x_{k-1})$ in the tangent space at $x_{k-1}$ -- this is consistent with using the inverse retraction to generalize the difference of two points -- and by $-\alpha_{k-1} \T(\grad f(x_{k-1}))$ in the tangent space at $x_{k}$.

We note that \eqref{eq:riem_bb} and \eqref{eq:mf_bb} generally only coincide if the exponential retraction and parallel transport are used. The usual generalization transports everything into the tangent space at the current iterate and calculates metrics there, in order to minimize the amount of metric applications. In the metric-free approach, we split expressions into parts that comply with quantities that are generalizable metric-free and thus minimize the amount of necessary vector transports. This is analogous to the case explored in Example \ref{ex:1}.


\paragraph{Backtracking line search} For an iterate $x_k \in \M$ and a search direction $d_k \in \Tan_{x_k}\M$, a~backtracking line search usually starts with a trial step size $\alpha_0$ or an initial bracket $[\alpha_0, \beta_0]$ and repeatedly updates it, until it finds step size such that an acceptance (lines earch) condition is met. Armijo and nonmonotone backtracking only shrink the trial step size by a prescribed factor $\rho$ and thus do not rely on the metric anyways. Bracketing strategies like Hager-Zhang line search \cite{HagerZ2005} use quadratic interpolation by bisection of the derivative and thus for a trial step size $\alpha$ need to evaluate
\begin{subequations}
    \begin{align}
        \label{eq:euc_slope}
        \frac{\drm}{\drm t} f(\R_{x_k}(td_t)) \Bigg|_{t = \alpha}  &= \nabla f(x_k + \alpha d_k)^Td_k, \\
        \label{eq:riem_slope}
        \frac{\drm}{\drm t}f(\R_{x_k}(td_t)) \Bigg|_{t = \alpha}  &= \langle \grad f( \R_{x_k} (\alpha d_k)), \T_{\alpha d_k}(d_k) \rangle_{\R_{x_k} (\alpha d_k)}, \\
        \label{eq:mf_slope}
        \frac{\drm}{\drm t} f(\R_{x_k}(td_t)) \Bigg|_{t = \alpha} &= \Drm f(\R_{x_k} (\alpha d_k))[\T_{\alpha d_k}(d_k)],        
    \end{align}
\end{subequations}
with the differentiated retraction transport from \eqref{eq:transport_diffret_curv}.

\subsection{Riemannian conjugate gradient} 

The Riemannian conjugate gradient methdod, RCG in short, is a simple improvement upon Euclidean Riemannian gradient descent. The search directions are chosen to be ``conjugate'' to the previous ones, often leading to faster convergence in practice. The search direction reads 
\begin{equation*}
    d_{k+1} = - g_{k+1} + \beta_{k+1} \T_{\alpha d_k}(d_k)
\end{equation*}
for $g_{k+1} = \grad f(x_{k+1})$ and the vector transport $\T$. There exists a wide variety for the choice conjugate gradient parameters $\beta$, thus we limit ourselves to the ones listed in \cite{Sato22}, but stress that the CG parameters compatible with the metric-free framework are not limited to those -- specifically hybrid parameters as discussed in \cite{SaIi21} can be adopted in a straight forward fashion. In the following, we abbreviate the transported gradient by $\T(g_k)$ and the transported search direction by $\T(d_k)$.

\paragraph{Hestenes--Stiefel} The Hestenes--Stiefel (HS) CG parameter \cite{HS52} is given by
\begin{subequations}\label{eq:param_hs}
    \begin{align}
        \beta_{k+1}^{\rm HS} &= \frac{g_{k+1}^T y_{k+1}}{\eta_k^T y_{k+1}},
            && y_{k+1} = g_{k+1} - g_k, \\
        \beta_{k+1}^{\rm HS} &= \frac{\langle g_{k+1}, y_{k+1}\rangle_{x_{k+1}}}
            {\langle g_{k+1}, \T(\eta_k)\rangle_{x_{k+1}} - \langle g_k, \eta_k\rangle_{x_k}},
            && y_{k+1} = g_{k+1} - \T(g_k), \\
        \beta_{k+1}^{\rm HS} &= \frac{\Drm f(x_{k+1})[y_{k+1}]}
            {\Drm f(x_{k+1})[\T(\eta_k)] - \Drm f(x_k)[\eta_k]},
            && y_{k+1} = g_{k+1} - \T(g_k).
    \end{align}
\end{subequations}

\paragraph{Fletcher--Reeves} The Fletcher--Reeves (FR) CG parameter \cite{FR64} is given by
\begin{subequations}\label{eq:param_fr}
    \begin{align}
        \beta_{k+1}^{\rm FR} &= \frac{\|g_{k+1}\|^2}{\|g_k\|^2}, \\
        \beta_{k+1}^{\rm FR} &= \frac{\|g_{k+1}\|_{x_{k+1}}^2}{\|g_k\|_{x_k}^2}, \\
        \beta_{k+1}^{\rm FR} &= \frac{\Drm f(x_{k+1})[g_{k+1}]}{\Drm f(x_k)[g_k]}.
    \end{align}
\end{subequations}

\paragraph{Polak--Ribi\`ere--Polyak} The Polak--Ribi\`ere--Polyak (PRP) CG parameter \cite{PR69,Pol69} is given by
\begin{subequations}\label{eq:param_prp}
    \begin{align}
        \beta_{k+1}^{\rm PRP} &= \frac{g_{k+1}^T y_{k+1}}{\|g_k\|^2},
            && y_{k+1} = g_{k+1} - g_k, \\
        \beta_{k+1}^{\rm PRP} &= \frac{\langle g_{k+1}, y_{k+1}\rangle_{x_{k+1}}}{\|g_k\|_{x_k}^2},
            && y_{k+1} = g_{k+1} - \T(g_k), \\
        \beta_{k+1}^{\rm PRP} &= \frac{\Drm f(x_{k+1})[y_{k+1}]}{\Drm f(x_k)[g_k]},
            && y_{k+1} = g_{k+1} - \T(g_k).
    \end{align}
\end{subequations}

\paragraph{Dai--Yuan} The Dai--Yuan (DY) CG parameter \cite{DY99} is given by
\begin{subequations}\label{eq:param_dy}
    \begin{align}
        \beta_{k+1}^{\rm DY} &= \frac{\|g_{k+1}\|^2}{\eta_k^T(g_{k+1} - g_k)}, \\
        \beta_{k+1}^{\rm DY} &= \frac{\|g_{k+1}\|_{x_{k+1}}^2}
            {\langle g_{k+1}, \T(\eta_k)\rangle_{x_{k+1}} - \langle g_k, \eta_k\rangle_{x_k}}, \\
        \beta_{k+1}^{\rm DY} &= \frac{\Drm f(x_{k+1})[g_{k+1}]}
            {\Drm f(x_{k+1})[\T(\eta_k)] - \Drm f(x_k)[\eta_k]}.
    \end{align}
\end{subequations}

\paragraph{Hybrid} Hybrid parameters take the form
\begin{equation}\label{eq:param_hybrid}
   \beta_{k+1}^{\rm Hybrid} = \max\{\min \{ \beta_{k+1}^{\rm A}, \beta_{k+1}^{\rm B}\}, 0\}
\end{equation}
for parameters $\beta_{k+1}^{\rm A}, \beta_{k+1}^{\rm B}$ with ${\rm A}, {\rm B} \in \{{\rm HS},{\rm FR},{\rm PRP},{\rm DY},\dots\}$
and are thus generalizable to the metric-free framework.

\section{Outlook: Optimization without Riemannian metric.}

Although a lot of quantities from Riemannian optimization can be expressed as solely depending on the derivative, as has been shown in the previous sections, there are notable exceptions, which prevent a straight forward translation of optimization algorithms to non-Riemannian manifolds. This includes
\begin{itemize}
    \item the Riemannian gradient and related search directions,
    \item the Riemannian Hessian and the related Riemannian Newton search direction.
\end{itemize}
In this section, we consider two non-Riemannian cases, where these quantities exist and then discuss how optimization methods can be constructed.

\subsection{Banach manifolds}

A \emph{Banach manifold} is a manifold modeled on a Banach space. In contrast to a Finsler manifold, it carries no canonical norm on its tangent bundles that varies continuously.

\subsubsection{Newton's method}

Without a metric, the notion of a Riemannian Hessian as a symmetric map on the tangent space does no longer make sense. Thus, the naive approach of approximating the Riemannian Hessian $\mathrm{Hess}\ f(x)$ and calculating $\mathrm{Hess}\ f(x)^{-1}[\grad_\M f(x)]$ for a Newton search direction is not possible.

The approach presented here follows the construction from \cite{WeiS2025B, WeiS2025A}, where the interested reader is referred to. We only present a brief overview here. Consider a smooth map ${f \colon \M \to \mathbb R}$, then for any extremum $x\in\M$ of $f$ it holds that $\Drm f(x) = 0_x^\ast \in
\Tan^\ast _x\M$, where $\Tan^\ast _x\M$ denotes the co-tangent space. We are thus looking for a zero of the map
\begin{equation*}
    \begin{aligned}
        \Drm f \colon \M \to \Tan^\ast\M
    \end{aligned}
\end{equation*}
on the co-tangent bundle. This is more involved than in the Euclidean case, since the space where the right hand side lives now depends on the unknown $x$. Vector-back transports
\[\overset{\leftarrow}{V_y}(x)\colon\Tan^\ast_x\M \to \Tan^\ast_y\M,\]
which describe how neighboring co-tangent spaces are related, are used to induce
a \emph{connection map} $Q$ for all $v \in \Tan^\ast\M$ with
\[Q_v \colon \Tan_v\Tan^\ast\M \to \Tan^\ast_{p(v)} \M,\]
where $p(v) \in \M$ defines the foot point of $v$. Constructing the linear mapping at $x \in \M$
\[ Q_{\Drm f(x)} \circ \Drm^2 f(x) \colon \Tan_x\M \to \Tan_x^\ast \M\]
gives us the well-defined, but not necessarily solvable \emph{Newton equation} for the Newton direction $\delta x \in \Tan_x\M$
\[Q_{\Drm f(x)} \circ \Drm^2 f(x)\ (\delta x) + \Drm f(x) = 0_x^\ast.\]

\subsubsection{Optimization on Banach manifolds}

On Banach manifolds, the derivative of the objective is well-defined, and, as long as a vector-back transport is provided, one has access to a Newton search direction. The following quantities are thus compatible with the Banach manifold setting.

\begin{itemize}
    \item \textbf{Search direction conditions:}
    \begin{itemize}
        \item Descent condition
    \end{itemize}
    \item \textbf{Line search conditions:}
    \begin{itemize}
        \item Armijo condition
        \item Goldstein condition
        \item Curvature conditions
    \end{itemize}
    \item \textbf{Step size strategies:}
    \begin{itemize}
        \item Backtracking line search
    \end{itemize}
\end{itemize}
Consequently, one can check if the Newton direction is a descent direction, and can use backtracking or bracketing strategies for the line search, satisfying some line search condition. 

In \cite{WeiS2025A}, under some conditions, local superlinear convergence rates for this variant of Newton's method (with constant step size $1$) are established.


\subsection{Finsler manifolds}
A {Finsler manifold} is a tuple $(\M, \|\cdot\|_\_)$ where $\|\cdot\|_x$ is a norm on $\Tan_x\M$ smoothly depending on $x \in \M$. In the infinite dimensional setting, we require the tangent space to be complete with respect to the norm, meaning $(\Tan_x\M, \|\cdot\|_x)$ is a Banach space.

\subsubsection{Finsler gradient}
Following the Golomb-Tapia definition \cite{GT72} we can generalize the notion of gradients to Banach spaces. For a Banach space $E$ and a Fréchet differentiable function ${f\colon U \to \mathbb R}$, $U\subset E$ the \emph{metric gradient} at $x \in U$ is given by 
\begin{equation}\label{eq:metric_gradient}
    \nabla_E f(x) = \mathop{\mathrm{argmax}}\limits_{\eta \in E, \|\eta\| \le \|\Drm f(x)\|_{x}} \Drm f(x)[\eta].
\end{equation}
On Hilbert spaces the metric gradient corresponds to the gradient, meaning
$$\nabla_Ef(x) = \{ \nabla f(x)\}.$$
It is shown in \cite{GT72} that $\nabla_E f(x)$ is a closed convex subset of $E$. If $E$ is reflexive, $\nabla_E f(x)$ is not empty, and if $E$ is strictly convex (meaning the unit ball is strictly convex), then $\nabla_E f(x)$ consists of at most one point. This notion of gradients is used in \cite{Penot2002} to define a descent algorithm, and it is shown that any cluster point of a sequence defined by a steepest descent algorithm in a general normed vector space is a critical point.

Generalizing \eqref{eq:metric_gradient} to Finsler manifolds is straight forward. 
Since the name metric gradient clashes with the naming of the (Riemannian) metric, we propose to name it \emph{Finsler gradient} in order to avoid confusion. This is analogous to naming it Riemannian gradient on a Riemannian manifold. For a Finsler manifold  $(\M, \|\cdot\|)$ and a Fréchet differentiable function ${f\colon\M\to\mathbb R}$, define the \emph{Finsler gradient} of $f$ at $x \in \M$ 
\begin{equation}
    \grad_\M f(x) = \mathop{\mathrm{argmax}}\limits_{\eta \in \Tan_x\M, \|\eta\|_x \le \|\Drm f(x)\|_{x}} \Drm f(x)[\eta].
\end{equation}
The Finsler gradient corresponds to the Riemannian gradient if $\M$ is a Riemannian manifold.
It is easy to see that the proof from \cite{GT72} still holds in the Finsler manifold case. Consequently $\grad_\M f(x)$ is a closed convex subset of $\Tan_x\M$. If $\Tan_x\M$ is reflexive, $\grad_\M f(x)$ is not empty, and if $\Tan_x\M$ is strictly convex, then $\grad_\M f(x)$ consists of at most one point. 

\subsubsection{Optimization on Finsler manifolds}

On Finsler manifolds, the gradient, a norm on the tangent space and the derivative of the objective are well-defined. Since every Finsler manifold is a Banach manifold, one additionally has access to a Newton search direction, if a vector-back transport is provided. Thus, all previously discussed quantities are well-defined in this setting. This means exactly the following quantities.
\begin{itemize}
    \item \textbf{Search direction conditions:}
    \begin{itemize}
        \item Descent condition
        \item Sufficient descent condition
        \item Angle condition
    \end{itemize}
    \item \textbf{Line search conditions:}
    \begin{itemize}
        \item Armijo condition
        \item Goldstein condition
        \item Curvature conditions
    \end{itemize}
    \item \textbf{Step size strategies:}
    \begin{itemize}
        \item Polyak step size
        \item Barzilai-Borwein step size
        \item Backtracking line search
    \end{itemize}
    \item \textbf{Riemannian conjugate gradient:}
    \begin{itemize}
        \item Hestenes-Stiefel parameter
        \item Fletcher-Reeves parameter
        \item Polak-Ribiere-Polyak parameter
        \item Dai-Yuan parameter
        \item Hybrid parameters
    \end{itemize}
\end{itemize}

Thus, using the Finsler gradient, one can construct gradient descent and conjugate gradient methods on Finsler manifolds, since all the relevant building blocks are available. For Newton-type methods, more involved search direction conditions, like the angle condition can be applied.

It also stands to reason that the convergence proof for these type of gradient descent methods from \cite{Penot2002} is applicable to the Finsler manifold generalization as well, however a rigorous proof for that is out of scope for this paper.

\section{Numerical experiments}

In order to illustrate effectiveness of metric-free methods, we consider a simple eigenvalue problem on the sphere, equipped with a non-standard metric, which we solve via Riemannian conjugate gradient. The numerical results show a significant reduction of CPU time and matrix-vector multiplications by using the metric-free approach. 

\subsection{Model problem}

For $m \in \mathbb N$, $n = m^2$ we consider the 2D Laplace matrix of an an $m\times m$  finite differences discretization grid
$$A = L \otimes I_m + I_m \otimes L \in \mathbb R^{n \times n}, \qquad L = \begin{pmatrix}
    2 & - 1 &  & & & \\
    -1 & 2 & -1 & & \\
     & -1 & \ddots & \ddots & \\
     & & \ddots & \ddots & -1\\
     & & & -1 & 2
\end{pmatrix} \in \mathbb R^{m \times m}.$$
It is well-known that $A$ is symmetric positive definite and the smallest eigenvalue goes to zero when $m$ goes to infinity. We are looking for a minimizer of
\begin{equation}\label{eq:min_problem}
    \min \limits_{ x \in \Sph } E(x) = \frac{1}{2} x^T A x,
\end{equation}
on the unit sphere
$$\Sph = \{ x \in \mathbb R^n \mid \|x\|_2^2 = 1\},$$
meaning we are looking for an eigenvector of the smallest eigenvalue of $A$. The tangent space at $x \in \Sph$ is given by
$$ \Tan_x \Sph = \{v \in \mathbb R^n  \mid v^Tx = 0\}$$
In order to solve \eqref{eq:min_problem} via Riemannian optimization, we equip $\Sph$ it with the metrics
$$\langle v,w\rangle_{2, x} = v^Tw, \quad \langle v,w\rangle_{M, x} = v^TMw, \qquad v,w \in \Tan_x\Sph,$$
where $M \approx A$ is a limited memory $LDL^T$ decomposition. The Euclidean $\langle \cdot, \cdot \rangle_{2,x}$-Riemannian gradient of \eqref{eq:min_problem} is given by
\begin{equation} \label{eq:residual}
    \grad_2 E(x) = Ax - x \,x^TAx
\end{equation}
The $\langle \cdot, \cdot \rangle_{M,x}$-Riemannian gradient of \eqref{eq:min_problem} is given by
\begin{equation}\label{eq:riem_grad_e}
   \grad_M E(x) = M^{-1}r(x) - M^{-1}x \frac{x^TM^{-1}r(x)}{x^T M^{-1}x},\qquad r(x) = \grad_2 E(x).
\end{equation}
This is the usual setting for metric-free Riemannian optimization, analogous to the example shown in Section 3. We have 
\begin{itemize}
    \item a ``cheap'' metric $\langle v,w\rangle_{2, x} $ to calculate the relevant quantities like step size and CG parameter and
    \item an ``expensive'' metric  $\langle v,w\rangle_{M, x} $ that is only used for the gradient and should never be applied explicitly. 
\end{itemize}

In order to solve \eqref{eq:min_problem}, we employ a Riemannian conjugate gradient method based on the {$\langle \cdot, \cdot \rangle_{M,x}$-gradient}, FR-PRP hybrid CG-parameter \eqref{eq:param_fr}, \eqref{eq:param_prp}, \eqref{eq:param_hybrid} and the cubic bracketing line search \cite{HagerZ2005}, satisfying the Wolfe conditions, comprised of the Armijo condition \eqref{eq:armijo} with $c_1 = 10^{-4}$ and strong curvature condition \eqref{eq:strong_curv} with $c_2 = 0.2$. We illustrate the superiority of metric-free approaches by comparing the checking of the curvature condition classically and metric-free.
\begin{figure*}[ht]

\begin{minipage}{0.48\textwidth}
\small
\begin{algorithm2e}[H]
\SetKwInOut{Input}{input}
\SetAlgoLined
\DontPrintSemicolon 
\Input{point $x$, direction $d$, step size $\alpha$, $M$-gradient $g$, $2$-gradient $r$}
\BlankLine
$\hat x = \R_x(\alpha d)$ \;
$\hat r = A\hat x - \hat x \hat x ^TA\hat x $ \;
$\hat g = M^{-1}  \hat{r} - M^{-1} \hat x \frac{\hat x^T M^{-1} \hat r}{\hat x M^{-1} \hat x} $\;
\KwRet{$|\mathcal T_{\alpha d}(d)^TM\hat g| \le c_2 d ^T M g$}
 \caption{Checking curvature condition classically}
     \label{alg:curv1}
\end{algorithm2e}
\end{minipage}\hfill
\begin{minipage}{0.48\textwidth}
\small
\begin{algorithm2e}[H]
\SetKwInOut{Input}{input}
\SetAlgoLined
\DontPrintSemicolon 
\Input{point $x$, direction $d$, step size $\alpha$, $M$-gradient $g$, $2$-gradient $r$}
\BlankLine
$\hat x = \R_x(\alpha d)$ \;
$\hat r = A\hat x - \hat x \hat x ^TA\hat x $ \;
\phantom{$\hat g = M^{-1}  \hat{r} - M^{-1} \hat x \frac{\hat x^T M^{-1} \hat r}{\hat x M^{-1} \hat x} $}\;
\KwRet{$|\mathcal T_{\alpha d}(d)^T\hat r| \le c_2 d ^T r $}
 \caption{Checking curvature condition metric-free\phantom{y}}
     \label{alg:curv2}
\end{algorithm2e}

\end{minipage}

\end{figure*}
In Algorithms \ref{alg:curv1} and \ref{alg:curv2} the checking of the strong curvature condition is outlined. Both procedures are mathematically equivalent, but differ significantly in the computational costs. In case the step size $\alpha$ is rejected and the gradient $\hat g$ cannot be re-used in the following iteration, the metric-free variant is two multiplications and two inversions of $M$ cheaper, since not only does the classical variant use the $M$-metric twice, it also needs to calculate the $M$-gradient $\hat g$. Even if $\hat g$ can be re-used, the metric-free appraoch is still two multiplications of $M$ cheaper. Similar effects can be observed for other building blocks of the RCG algorithms.

\subsection{Numerical results}

The experiment was implemented in \texttt{julia} using \texttt{Manopt.jl}\cite{Manopt2022}.
The source code is available at
\begin{center}
    \url{https://github.com/jonas-pueschel/MetricFreeTest}
\end{center}
The discretization was chosen as $m= 500$, thus $n = 250\,000$. The preconditioner $M = LDL^T$ was implemented using the \texttt{LimitedLDLFactorization.jl}\cite{Orban2019-yj} with \texttt{memory = 10}. 
We tested the regular $M$-Riemannian conjugate gradient (M-RCG) and metric-free $M$-Riemannian conjugate gradient (M-RCG-mf), which are mathematically equivalent. For illustrative purpose, we also added Euclidean inner product Riemannian conjugate gradient (2-RCG). As performance metrics, we collected the CPU-time and the amount of matrix-vector multiplications with $A$, $M$ and $M^{-1}$. The  convergence condition is  $E-E_{\rm ref} < 10^{-15}$, where the reference energy $E_{\rm ref}$ was calculated using (M-RCG) with a tolerance of $\| \grad_ME(x)\|_M \le 10^{-12}$.

\begin{figure}
    \centering
\input{test_plot}
    \caption{Convergence plots of the three methods with respect to CPU time (left) and matrix-vector multiplications (right). The metric-free variant outperforms the regular method significantly in all aspects.}
    \label{fig:placeholder}
\end{figure}

\begin{table}
    \centering
    \input{test_table}
    \caption{Comparison of the different methods. While M-RCG and M-RCG-mf are mathematically equivalent, the metric-free method only needs a fraction of $M^{-1}$ matrix-vector multiplications and no $M$ matrix-vector multiplications, resulting in a CPU time reduction of around 82\% and a reduction in matrix-vector multiplications of around 78\%.}
    \label{tab:placeholder}
\end{table}

As can be seen in Figure \ref{fig:placeholder} and Table \ref{tab:placeholder}, using the metric-free approach (M-RCG-mf) results in a CPU time reduction of around 82\% and a reduction in matrix-vector multiplications of around 78\% compared to (M-RCG). Since the methods are mathematically equivalent, there are no downsides of using the metric-free variant, and both (M-RCG) and (M-RCG-mf) take 230 iterations to converge, where the only difference between iterates are numerical inaccuracies in floating point arithmetric. Consequently, there are no downsides of using the metric-free variant, while the performance improves significantly.

\section{Conclusion}
In this paper, we formulated metric-free versions of many quantities in the framework of Riemannian optimization, allowing a formulation of Riemannian optimization methods where the Riemannian metric is never applied explicitly. While keeping mathematical accuracy and reducing numerical errors, these modifications show significant performance improvements in practice, when expensive non-standard Riemannian metrics are used. 

\bigskip\noindent
\textbf{Acknowledgments} The author thanks Ronny Bergmann for fruitful and constructive discussions on mathematical aspects of metric-free Riemannian optimization, as well as the joint implementation of metric-free Riemannian optimization into \texttt{Manopt.jl}.

\bibliographystyle{plain}
\bibliography{bibliography.bib}

\appendix

\end{document}

%% file: preamble.tex
\usepackage[toc,page]{appendix}
\usepackage{graphicx} 
\usepackage{fix-cm} 
\usepackage{anyfontsize}
\usepackage{url}
\usepackage{amsmath} 
\usepackage{amssymb}
\usepackage{bm, bbm}
\usepackage{polynom}
\usepackage[outline]{contour}
\usepackage{xcolor}
\usepackage{pgfplots}
\usepackage{mathtools}
\usepackage{stmaryrd}
\usepackage[labelfont=rm]{subcaption}
\pgfplotsset{
  compat=1.13,
}
\usepackage{multicol}
\usepackage{multirow}
\usepackage{xfrac}  
\usepackage{enumitem} 
\usepackage{float}
\usepackage{comment}
\usepackage{mathrsfs}
\usepackage[ruled,vlined, algo2e, linesnumbered]{algorithm2e}

\usepackage{tikz}
\usepackage{tikz-cd}
\usepackage{scalerel,stackengine}

\usepackage{natbib}
\usepackage{rotating}
\usepackage{pdflscape}

\theoremstyle{plain}
\newtheorem{thm}{Theorem}[section] 

\newtheorem{lem}[thm]{Lemma}
\newtheorem{rem}[thm]{Remark}

\makeatletter
\renewcommand\paragraph{\@startsection{paragraph}{4}{\z@}%
  {-3.25ex \@plus -1ex \@minus -.2ex}%
  {-1em}%
  {\normalfont\bfseries}}
\makeatother

\newtheorem{exm}[thm]{Example} 

\contourlength{0.1pt}
\contournumber{10}%
\definecolor{clr1}{RGB}{27,158,119}
\definecolor{clr2}{RGB}{217,95,2}
\definecolor{clr3}{RGB}{117,112,179}
\definecolor{clr4}{RGB}{231,41,138}
\definecolor{clr5}{RGB}{102,166,30}
\definecolor{clr6}{RGB}{230,171,2}
\definecolor{clr7}{RGB}{166,118,29}
\definecolor{clr8}{RGB}{102,102,102}

\newcommand{\E}{\mathcal{E}}

\newcommand{\M}{\mathcal{M}}

\newcommand{\R}{\mathcal{R}}

\newcommand{\T}{\mathcal{T}}

\newcommand{\Sph}{\mathcal{S}^{n-1}}
\newcommand{\Tan}{\mathrm{T}}

\newcommand{\IR}{{\mathbb R}}

\newcommand{\grad}{\mathrm{grad}}

\newcommand{\Drm}{\mathrm{D}\,}
\newcommand{\drm}{\mathrm{d}}

\makeatletter
\newsavebox{\@brx}
\newcommand{\llangle}[1][]{\savebox{\@brx}{\(\m@th{#1\langle}\)}%
  \mathopen{\copy\@brx\mkern2mu\kern-0.9\wd\@brx\usebox{\@brx}}}
\newcommand{\rrangle}[1][]{\savebox{\@brx}{\(\m@th{#1\rangle}\)}%
  \mathclose{\copy\@brx\mkern2mu\kern-0.9\wd\@brx\usebox{\@brx}}}
\makeatother

\usepackage[most]{tcolorbox}

\definecolor{myGreen2}{RGB}{114,175,30} 
\definecolor{myRed}{RGB}{180,50,50}  
\definecolor{myOrange}{RGB}{225,92,22} 

\definecolor{unia-purple}{RGB}{173, 0, 124}
\definecolor{light-gray}{rgb}{0.8889,0.8889,0.8889}
\definecolor{cpl1}{rgb}{0.8889,0.4356,0.2781}
\definecolor{cpl2}{rgb}{0.0,0.6056,0.9787}
\definecolor{cpl3}{rgb}{0.2422,0.6433,0.3044}
\definecolor{cpl4}{rgb}{0.2, 0.2, 0.2}

\newlength{\plotwidth}
\newlength{\plotheight}
\usetikzlibrary{arrows.meta}
\usetikzlibrary{backgrounds}
\usepgfplotslibrary{patchplots}
\usepgfplotslibrary{fillbetween}
\pgfplotsset{%
    layers/standard/.define layer set={%
        background,axis background,axis grid,axis ticks,axis lines,axis tick labels,pre main,main,axis descriptions,axis foreground%
    }{
        grid style={/pgfplots/on layer=axis grid},%
        tick style={/pgfplots/on layer=axis ticks},%
        axis line style={/pgfplots/on layer=axis lines},%
        label style={/pgfplots/on layer=axis descriptions},%
        legend style={/pgfplots/on layer=axis descriptions},%
        title style={/pgfplots/on layer=axis descriptions},%
        colorbar style={/pgfplots/on layer=axis descriptions},%
        ticklabel style={/pgfplots/on layer=axis tick labels},%
        axis background@ style={/pgfplots/on layer=axis background},%
        3d box foreground style={/pgfplots/on layer=axis foreground},%
    },
}

%% file: test_table.tex
\begin{tabular}{l|r|r|r|r|r|r}
    method & CPU time & iterations & $A$ matvec& $M$ matvec& $M^{-1}$ matvec& total matvec\\ \hline 
    2-RCG & 54.04 s & 3804 & 8603 & 0 & 0 & 8603 \\ 
M-RCG & 34.80 s & 230 & 316 & 2158 & 1092 & 3566 \\ 
M-RCG-mf & 6.42 s & 230 & 316 & 0 & 460 & 776 \\ 
\end{tabular}